\documentclass[11pt,a4paper]{article}
\usepackage[latin1]{inputenc}
\usepackage{amsmath}
\usepackage{amsthm}
\usepackage{amsfonts}
\usepackage{amssymb,url}
\usepackage{graphicx}
\usepackage[margin=2.7cm]{geometry}
\usepackage{xcolor}
\newtheorem{theorem}{Theorem}
\newtheorem{lemma}{Lemma}

\author{Mahadi Ddamulira, Paul Emong$ ^{*} $, and Geoffrey Ismail Mirumbe}
\title{Members of Narayana's cow sequence that are concatenations of two repdigits}
\date{}

\begin{document}
\maketitle

\begin{abstract}
\noindent Let $ (N_n)_{n\ge 0}$ be the Narayana's cow sequence defined by a third-order recurrence relation $ N_0=0,\ N_1= N_2=1 $, and $ N_{n+3}=N_{n+2}+N_n $ for all $ n\ge 0 $. In this paper, we determine all Narayana numbers that are concatenations of two repdigits. The proof of our main theorem uses lower bounds for linear forms in logarithms and a version of the Baker-Davenport reduction method in Diophantine approximation.
\end{abstract}



\noindent 
\thanks{$ ^{*} $ Corresponding author}

\section{Introduction.}

\noindent In $1356$, the Indian mathematician Narayana Pandit wrote his famous book titled
Ganita Kaumudi where he proposed the following problem of a herd of cows and calves: A cow produces one calf every year. Beginning in its fourth year, each calf produces one calf at the beginning of each year. How many calves are there altogether after 20 years? \cite{Alou}.
\medskip

\noindent This problem is translated into modern language of recurrence sequences. We observe that the number of cows increased by one after one year, increased by one after two years, increased by one after three years and increased by two after four years and so on. Hence, we obtain the sequence $1, 1, 1, 2,\cdots$. In the $n$-th year, Narayana's cow sequence problem can be written as the following linear recurrence sequence:

$$N_{n+3} = N_{n+2} + N_{n},$$ \text{for all $n \geq 0$, with the initial conditions $N_{0} = 0$, and $N_{1} =N_{2} = 1$}.
Its first few terms are
\begin{align*}
(N_n)_{n \geq 0} = \{0,1,1,1,2,3,4,6,9,13,19, \ldots\}.
\end{align*}
\medskip

\noindent A \textit{repdigit} is a positive integer $N$ that has only one distinct digit when written in base 10. That is, $N$ is of the form
\begin{align*}
N = \overline{\underbrace{d\cdots d}_\text{$\ell$ times}} = d \left( \frac{10^\ell - 1}{9} \right),
\end{align*}
for some non negative integers $d$ and $\ell$ with $1 \leq d \leq 9$ and $\ell \geq 1$. 
\medskip

\noindent The problem of finding repdigits in a linear recurrence sequence has been
studied. For example, in \cite{MD2, Dda}, Ddamulira studied the problem of finding Padovan and tribonacci numbers which are concatenations of two repdigits, respectively. In \cite{Ala}, the authors proved that the only Fibonacci numbers that are concatenations of two repdigits are $\{13, 21, 34, 55, 89, 144, 233, 377\}$. Erduvan and Keskin, in their paper obtained all Lucas numbers which are concatenations of two repdigits, see \cite{Eduv1}. In \cite{Eduv2}, the same authors determined Lucas numbers which are concatenations of three repdigits. Furthermore, Rayaguru and Panda \cite{Ray} showed that 35 is the only balancing number that can be written as a concatenation of two repdigits. More related results in this direction include: the result of Qu and Zeng \cite{Q}, the result of Bravo et al. \cite{Brav2} and the result of Trojovsk\'{y} \cite{Tr}. 
\medskip

\noindent Continuing in the same direction of research, we study the problem of writing all Narayana numbers which are concatenations of two repdigits. To be precise, we find all solutions of the Diophantine equation
\begin{align}\label{eq1}
N_n = \overline{\underbrace{d_1\cdots d_1}_\text{$m_{1}$ times} \underbrace{d_2\cdots d_2}_\text{$m_{2}$ times} },
\end{align}
in non-negative integers $n,d_{1},d_{2},m_{1},m_{2}$ with $n\geq 0,m_{1}\geq m_{2}\geq 1$ and $d_1 , d_2 \in \{0,1, 2, 3,\ldots, 9 \},\\d_{1}>0$.
\medskip

\noindent Thus, the main result is the following:
\begin{theorem}\label{thm1x}
The only Narayana numbers which are concatenations of two repdigits are 
\begin{align*}
N_n \in \{13,19,28,41,60,88,277\}.
\end{align*}

\end{theorem}

\section{Preliminaries.}
Here, we state some facts about the Narayana's cow sequence, theorem and lemmas that are relevant in the proof of the main results.

\subsection{Some properties of the Narayana's cow sequence.}
First, the Binet's formula for the Narayana sequence is written as
\begin{align}\label{eq2}
	N_n = a\alpha^n + b\beta^n + c\gamma^n.
\end{align}
for all $n \geq 0$,
The characteristic equation of the Narayana's cow sequence is given by\\ $\varphi (x) := x^3 - x^2 - 1=0$, having roots $\alpha (\approx 1.46557),~\beta~\text{and}~\gamma = \overline{\beta}$ with $|\beta|=|\gamma|<1$, where
\begin{align*}
a = \frac{\alpha}{(\alpha-\beta)(\alpha-\gamma)} ,~b=\frac{\beta}{(\beta-\alpha)(\beta-\gamma)} ~\text{and}~ c= \frac{\gamma}{(\gamma-\alpha)(\gamma-\beta)}.
\end{align*}
In addition, $$a=\displaystyle\frac{\alpha^{2}}{\alpha^{3}+2},$$ and its minimal polynomial over integers is $31x^{3}-3x-1$.
Setting
\begin{align*}
t(n) := N_n - a\alpha^n =  b\beta^n + c\gamma^n,~ \text{we notice that}~ |t(n)| < \frac{1}{\alpha^{n/2}} \quad \text{for all} \quad n \geq 1.
\end{align*}
Furthermore, it can be proved by induction, that
\begin{align}\label{eq3}
\alpha^{n-2} \leq N_n \leq \alpha^{n-1} ~\text{for all}~ n\geq 1.
\end{align}
Observe that the characteristic polynomial $\varphi(x)$ is irreducible in $\mathbb{Q}[x]$.
\noindent Let $\mathbb{K} := \mathbb{Q}(\alpha, \beta)$ be the splitting field of the polynomial $\varphi$ over $\mathbb{Q}$. Then $[\mathbb{K}: \mathbb{Q}] = 6$ and $[\mathbb{Q}(\alpha): \mathbb{Q}] = 3$. The Galois group of $\mathbb{K}/\mathbb{Q}$ is given by
\begin{align*}
\mathcal{G} := \text{Gal}(\mathbb{K}/\mathbb{Q}) \cong \{ (1), (\alpha \beta), (\alpha \gamma), (\beta \gamma), (\alpha \beta \gamma)  \} \cong S_3.
\end{align*}
We identify the automorphisms of $\mathcal{G}$ with the permutation group of the zeroes of $\varphi$. We highlight the permutation $(\alpha \beta)$, corresponding to the automorphism $\sigma : \alpha \mapsto \beta, \beta \mapsto \alpha, \gamma \mapsto \gamma$, which we use later to obtain a contradiction on the size of the absolute value of a certain bound.

\subsection{Linear forms in logarithms.}

\noindent To solve the Diophantine equations involving repdigits and the terms of binary recurrence sequences, many authors have used Baker's theory to reduce lower bounds concerning linear forms in logarithms of algebraic numbers. These lower bounds play an important role while solving such Diophantine equation. We start with recalling some basic definitions and results from algebraic
number theory.
\medskip

\noindent Let $\eta$ be an algebraic number of degree $d$ with minimal primitive polynomial over the integers
\begin{align*}
g(x) = a_0 \prod_{j = 1}^d (x - \eta^{(j)}),
\end{align*}
where the leading coefficient $a_0$ is positive and the $\eta^{j}$'s are the conjugates of $\eta$. We define the logarithmic height of $\eta$ by
\begin{align*}
h(\eta) := \frac{1}{d} \left( \log a_0 + \sum_{j = 1}^d \log \left( \max \{|\eta^{(j)}|, 1  \}  \right)   \right).
\end{align*}
Note that, if $\eta = \frac{p}{q} \in \mathbb{Q}$ with $gcd(p,q)=1$ and $q > 0$, then the above definition reduces to $$h(\eta) = \log \max \{ |p|,q \}.$$
The following are the  properties of the logarithmic height function, which will be used in the subsequent sections of this paper without reference:
\begin{align*}
h(\eta_1 \pm \eta_2) & \leq h(\eta_1) + h(\eta_2) + \log 2, \\
h(\eta_1 \eta_2 ^{\pm}) & \leq h(\eta_1) + h(\eta_2), \\
h(\eta^s) & = |s| h(\eta), \quad (s \in \mathbb{Z}).
\end{align*}
\noindent The following theorem is useful in obtaining the lower bound. We us the version of Baker's theorem proved by Bugeaud, Mignotte and Siksek (\cite{BMS}, Theorem 9.4).

\begin{theorem}[Bugeaud, Mignotte, Siksek, \cite{BMS}]\label{thm2}
Let $\eta_1, \ldots, \eta_t$ be positive real algebraic numbers in a real algebraic number field $\mathbb{K} \subset \mathbb{R}$ of degree $D$. Let $b_1, \ldots, b_t$ be nonzero integers such that 
\begin{align*}
\Gamma := \eta_1 ^{b_1} \ldots \eta_t ^{b_t} - 1 \neq 0.
\end{align*}
Then
\begin{align*}
\log | \Gamma| > - 1.4 \times 30^{t+3} \times t^{4.5} \times D^2 (1 + \log D)(1 + \log B)A_1 \ldots A_t,
\end{align*}
where
\begin{align*}
B \geq \max \{|b_1|, \ldots, |b_t| \},
\end{align*}
and
\begin{align*}
A_j \geq \max \{ D h(\eta_{j}), |\log \eta_{j}|, 0.16  \}, \quad \text{for all} \quad j = 1, \ldots,t.
\end{align*}
\end{theorem}

\noindent Next, we state the following Lemma that leads to an upper bound of $n$.
\begin{lemma}[G\'{u}zman S\'{a}nchez, Luca, \cite{GSL}]\label{l1}
	Let $r \geq 1$ and $H > 0$ be such that $H > (4r^2)^r$ and $H > L/(\log L)^r$. Then
	\begin{align*}
		L < 2^r H (\log H)^r.
	\end{align*}
\end{lemma}
\subsection{Reduction procedure.}
The bounds on the variables obtained via Baker's theorem are too large for any computational purposes. To reduce the bounds, we use Baker-Davenport reduction method \cite{Brav}, which is a variation of a result due to  Dujella and Peth\H{o} (\cite{DP}, Lemma 5a). For a real number $r$, we denote by $\parallel r \parallel$ the quantity $\min \{|r - n| : n \in \mathbb{Z} \}$, the distance from $r$ to the nearest integer.
\begin{lemma}[Dujella, Peth\H o, \cite{DP}]\label{l2}
Let $\kappa \neq 0, A, B$ and $\mu$ be real numbers such that $A > 0$ and $B > 1$. Let $M > 1$ be a positive integer and suppose that $\frac{p}{q}$ is a convergent of the continued fraction expansion of $\tau$ with $q > 6M$. Let 
\begin{align*}
\epsilon := \parallel \mu q \parallel - M \parallel \tau q \parallel.
\end{align*}
If $\epsilon > 0$, then there is no solution of the inequality
\begin{align*}
0 < |m \tau - n + \mu| < AB^{-k}
\end{align*}
in positive integers $m,n,k$ with
\begin{align*}
\frac{\log (Aq/\epsilon)}{\log B} \leq k \quad \text{and} \quad m \leq M.
\end{align*}
\end{lemma}
We also mention a known fact of the exponential function, which is stated as a Lemma for further reference.
\begin{lemma}\label{l3}
	For any non-zero real number $x$, we have
\begin{align*}
	\text{If}~ x<0 ~\text{and}~\mid e^{x}-1\mid<\frac{1}{2},~\text{then}~\mid x\mid<2\mid e^{x}-1\mid.
\end{align*}
\end{lemma}

\section{Proof of the Main Result.}
\subsection{The low range.}
A computer search for solutions of Diophantine equation $(1)$ in the ranges $0\leq d_{2}<d_{1}\leq 9$ and $1\leq m_{2}\leq m_{1}\leq n\leq 250$ gives all solutions listed in Theorem $1$. Now, we assume that $n > 250$.
\subsection{The initial bound on $n$.}
To begin with, we consider the Diophantine equation $(\ref{eq1})$, and rewrite it as
\begin{align}\label{eq4}
N_n & = \overline{\underbrace{d_1\cdots d_1}_\text{$m_{1}$ times} \underbrace{d_2\cdots d_2}_\text{$m_{2}$ times} } \nonumber \\
& = \overline{\underbrace{d_1\cdots d_1}_\text{$m_{1}$ times}}.10^{m_{2}} + \overline{ \underbrace{d_2\cdots d_2}_\text{$m_{2}$ times}} \nonumber \\
&=d_{1} \left( \frac{10^{m_{1}} - 1}{9} \right).10^{m_{2}}+d_{2} \left( \frac{10^{m_{2}} - 1}{9} \right)\nonumber\\
& = \frac{1}{9} \left(d_1.10^{m_{1} + m_{2}} - (d_1 - d_2). 10^{m_{2}} - d_2 \right).
\end{align}
Next, we state and prove the following lemma which relates the size of $n$ and $m_{1}+m_{2}$.
\begin{lemma}\label{l4}
All solutions of $\eqref{eq4}$ satisfy
\begin{align*}
(m_{1} + m_{2}) \log 10 - 2 < n \log \alpha < (m_{1} + m_{2}) \log 10 + 1.
\end{align*}
\end{lemma}
\begin{proof}
\noindent Using $(\ref{eq3})$ and $(\ref{eq4})$, we get 
\begin{align*}
\alpha^{n-2} \le N_n < 10^{m_{1} + m_{2}}.
\end{align*}
Taking the logarithm on both sides, we have that 
\begin{align*}
(n-2)\log\alpha<(m_{1}+m_{2})\log 10,
\end{align*}
leading to
\begin{align}\label{eq5}
n\log\alpha<(m_{1}+m_{2})\log 10+2\log\alpha<(m_{1}+m_{2})\log 10+1	
\end{align}
For the lower bound, we obtain
\begin{align*}
10^{m_{1}+m_{2}-1}<N_{n}\leq\alpha^{n-1},
\end{align*}
and taking logarithms on both sides, we get
\begin{align*}
	(m_{1}+m_{2}-1)\log 10<(n-1)\log\alpha,
\end{align*}
which results to
\begin{align}\label{eq6}
	(m_{1}+m_{2})\log 10-2<(m_{1}+m_{2}-1)\log 10+\log\alpha<n\log\alpha.
\end{align}
Comparing $(\ref{eq5})$ and $(\ref{eq6})$ completes the proof of Lemma \ref{l4}.
\end{proof}

\noindent Next, we proceed to investigate $(\ref{eq4})$ in the following two step:
\medskip

\noindent \textbf{Step 1.} Using $(\ref{eq2})$ and $(\ref{eq4})$, we have
\begin{align*}
(a\alpha^n + b \beta^n + c\gamma^n )= \frac{1}{9}(d_1.10^{m_{1} + m_{2}} - (d_1-d_2).10^{m_{2}} - d_2).
\end{align*}
Equivalently,
\begin{align*}
9 a\alpha^n  - d_1.10^{m_{1} + m_{2}} = -9t(n) - (d_1-d_2).10^{m_{2}} - d_2.
\end{align*}
Thus, we have that
\begin{align*}
|9 a\alpha^n  - d_1.10^{m_{1}+m_{2}}| & = |-9t(n) - (d_1-d_2).10^{m_{2}} -d_2| \\
& \leq 9\alpha^{-n/2}+9.10^{m_{2}}+9 \\
& < 28.10^{m_{2}},
\end{align*}
 where we used the fact that $n > 250$. Dividing both sides of the inequality by $d_1.10^{m_{1} + m_{2}}$ gives
\begin{align}\label{eqn7}
\left | \left( \frac{9a}{d_1} \right).\alpha^n.10^{-m_{1}-m_{2}} - 1  \right | < \frac{28. 10^{m_{2}}}{d_1.10^{m_{1} +m_{2}}} < \frac{28}{10^{m_{1}}}.
\end{align}

Let
\begin{align}\label{eq8}
\Lambda_1 := \left( \frac{9a}{d_1} \right). \alpha^n.10^{-m_{1} -m_{2}} - 1.
\end{align}
We then proceed to apply Theorem $\ref{thm2}$ on (\ref{eq8}). First, observe that $\Lambda_1 \neq 0$. If it were, then we would have that $$a \alpha^n = \frac{d_{1}}{9}.10^{m_{1}+m_{2}}.$$ In this case therefore, applying the automorphism $\sigma$ of the Galois group $\mathcal{G}$ on both sides of the preceeding equation and taking absolute values, we obtain
\begin{align*}
\left|\left( \frac{d_{1}}{9}.10^{m_{1}+m_{2}} \right)\right| = |\sigma (a \alpha^n)| = |c\gamma^n| < 1,
\end{align*}
which is false. Thus, we have $\Lambda_1 \neq 0$.
\medskip

\noindent Theorem $\ref{thm2}$ is then applied on $(\ref{eq8})$ with the following parameters:
\begin{align*}
\eta_1: = \frac{9a}{d_1}, \ \eta_2: = \alpha, \ \eta_3: = 10, \ b_1: = 1, \ b_2: = n, \ b_3: = -m_{1} - m_{2}, \ t: = 3.
\end{align*}
From Lemma $\ref{l4}$, we have that $m_{1} + m_{2} < n$. Consequently, we choose $B := n$. Notice that $\mathbb{K}: = \mathbb{Q}(\eta_1, \eta_2, \eta_3) = \mathbb{Q}(\alpha)$, since $a=\displaystyle\frac{\alpha^{2}}{\alpha^{3}+2}$. Therefore, $D: = [\mathbb{K}: \mathbb{Q}] = 3$.
\medskip

\noindent Using the properties of the logarithmic height, we estimate $h(\eta_{1})$ as follows:
\begin{align*}
h(\eta_1) &= h \left( \frac{9a}{d_1} \right)\\
 &\le h(9)+h(a)+h(d_{1})\\
&\leq \log 9+\frac{1}{3}\log 31+\log 9\\
&<2.41
\end{align*}
Similarly, we that $h(\eta_2) = h(\alpha) = \displaystyle\frac{\log \alpha}{3}$ and $h(\eta_3) = h(10)= \log 10$. Therefore, we choose
\begin{align*}
A_1: = 7.23, \ A_2: = \log \alpha,~\text{and}~ \ A_3: = 3 \log 10.
\end{align*}
By Theorem $\ref{thm2}$, we get
\begin{align*}
\log |\Lambda_1| &> -1.4(30^6)(3^{4.5})(3^2)(1 + \log 3)(1 + \log n)(7.23)(\log \alpha)(3 \log 10)\\& > -6.9\times 10^{12}(1 + \log n), 
\end{align*}
which when compared with $(\ref{eqn7})$ gives
\begin{align*}
m_{1} \log 10 - \log 28 < 6.9\times 10^{12}(1 + \log n),
\end{align*}
leading to
\begin{align}\label{eq9}
m_{1} \log 10 < 7.0\times 10^{12}(1 + \log n).
\end{align}

\noindent \textbf{Step 2.} Rewriting $(\ref{eq4})$, we obtain 
\begin{align*}
9 a\alpha^n  - (d_{1}.10^{m_{1}}-(d_{1}-d_{2})).10^{m_{2}}= -9t(n) - d_2,
\end{align*}
which shows that
\begin{align*}
|9a \alpha^n  - (d_1.10^{m_{1}} - (d_1-d_2)).10^{m_{2}}| & = |-9t(n) - d_2| \\
& \leq 9\alpha^{-n/2}+9<18.
\end{align*}
Dividing both sides of the inequality by $9 a\alpha^n$, we have that
\begin{align}\label{eq10}
\left| \left( \frac{d_1.10^{m_{1}} - (d_1 - d_2)}{9a} \right).\alpha^{-n}.10^{m_{2}}- 1 \right| & < \frac{18}{9 a\alpha^n}< \frac{2}{\alpha^n}
\end{align}
Now, we let
\begin{align*}
\Lambda_2 :=  \left( \frac{d_1.10^{m_{1}} - (d_1 - d_2)}{9a} \right).\alpha^{-n}.10^{m_{2}}- 1
\end{align*}
Using similar arguments as in $\Lambda_{1}$, we apply Theorem $\ref{thm2}$ on $\Lambda_{2}$. We notice that $\Lambda_{2}\neq 0$. If it were, then we would have that 
\begin{align*}
a\alpha^n = \left( \frac{d_1.10^{m_{1}} - (d_1 - d_2)}{9} \right).10^{m_{2}}.
\end{align*}
Applying the automorphism $\sigma$ of the Galois group $\mathcal{G}$ on both sides, and taking the absolute values, we obtain
\begin{align*}
\left| \left( \frac{d_1.10^{m_{1}} - (d_1 - d_2)}{9} \right).10^{m_{2}} \right| = |\sigma( a\alpha^n)| = | c\gamma^n | < 1,
\end{align*}
which is false. Therefore, $\Lambda_{2}\neq 0$. We then proceed to apply Theorem \ref{thm2} with the following parameters:
\begin{align*}
\eta_1 := \left(\frac{d_1.10^{m_{1}} - (d_1-d_2)}{9a} \right), \ \eta_2 := \alpha, \ \eta_3: = 10, \ b_1: = 1, \ b_2: = -n, \ b_3 := m_{2}, \ t: = 3.
\end{align*}
Since $m_{2}<n$, we take $B:=n$.
Again, taking $\mathbb{K}: = \mathbb{Q}(\eta_1, \eta_2, \eta_3) = \mathbb{Q}(\alpha)$, we have that $D: = [\mathbb{K}: \mathbb{Q}] = 3$. 
\medskip
Next, we use the properties of the logarithmic height to estimate $h(\eta_{1})$ as before, and obtain
\noindent 
\begin{align*}
h(\eta_1) & = h \left(\frac{d_1.10^{m_{1}} - (d_1-d_2)}{9a} \right)\\
&\leq h(d_{1}.10^{m_{1}}-(d_{1}-d_{2}))+h(9A)\\
&\leq h(d_{1}.10^{m_{1}})+h(d_{1}-d_{2})+h(9)+h(a)+\log 2\\
& \leq h(d_{1})+m_{1}h(10)+h(d_{1})+h(d_{2})+h(9)+h(a)+2\log 2\\
& \leq  m_{1}\log 10 +4\log 9+\displaystyle\frac{1}{3}31 + 2\log 2\\
& \leq 7.0\times 10^{12}(1 + \log n) + 4\log 9 +\displaystyle\frac{1}{3}\log 31 + 2\log 2\\
& < 7.1\times 10^{12}(1 + \log n).
\end{align*}
\medskip

\noindent We then take 
\begin{align*}
A_1: = 2.13 \times 10^{13} (1 + \log n),~A_2: = \log \alpha~\text{and}~A_3: = 3 \log 10. 
\end{align*}
\medskip

\noindent Theorem $\ref{thm2}$ says that
\begin{align*}
\log |\Lambda_2| & > -1.4(30^6)(3^{4.5})(3^2)(1 + \log 3)(1 + \log n)(2.13 \times 10^{13} (1 + \log n))(\log\alpha)(3\log 10)\\
& > ~-2.0 \times 10^{25}(1 + \log n)^2,
\end{align*}
and comparison of this inequality with $\eqref{eq10}$ gives
\begin{align}\label{eq7}
n \log \alpha -\log 2< 2.0 \times 10^{25}(1 + \log n)^2,
\end{align}
which simplifies to
\begin{align*}
n < 8.0 \times 10^{25}( \log n)^2.
\end{align*}
Next, we apply Lemma $\ref{l3}$ that enables us to find an upper bound of $n$, with following parameters:$$r: = 2,\ L: = n,\ \text{and} \ H: = 8.0 \times 10^{25}.$$ 
Therefore, we have
\begin{align*}
n < 2^2 (8.0 \times 10^{25})(\log 8.0 \times 10^{25})^2,
\end{align*}
which results to 
\begin{align*}
n < 2.15 \times 10^{29}.
\end{align*}
By Lemma $\ref{l4}$, we have that
\begin{align*}
m_{1} + m_{2} < 3.56 \times 10^{28}.
\end{align*}
The following lemma is a result of the proof that has been completed.
\begin{lemma}\label{l5}
All solutions to $(\ref{eq4})$ satisfy
\begin{align*}
m_{1} + m_{2} < 3.56 \times 10^{28} \quad \text{and} \quad n < 2.15 \times 10^{29}.
\end{align*}
\end{lemma}

\subsection{Reducing the bounds.}
\noindent Using $(\ref{eqn7})$, let
\begin{align*}
\Gamma_1: = - \log (\Gamma_1 + 1) = (m_{1} + m_{2})\log 10 - n\log \alpha - \log \left(  \frac{9a}{d_1} \right).
\end{align*}
 Notice that $(\ref{eqn7})$ is rewritten as
\begin{align*}
\left| e^{- \Gamma_1} - 1 \right| < \frac{28}{10^{m_{1}}}.
\end{align*}
Observe that $-\Gamma_1\neq 0$, since $e^{-\Gamma_1}-1=\Lambda_{1}\neq 0$.
 Assume that $m_{1}\ge 2$, then $$\left| e^{- \Gamma_1} - 1 \right| < \displaystyle\frac{7}{25}<\displaystyle\frac{1}{2}.$$ Therefore, by Lemma \ref{l3}, we have that $$\left|\Gamma_1 \right| < \frac{56}{10^{m_{1}}}.$$ 
Substituting $\Gamma_1$ in the above inequality with its value and dividing through by $\log \alpha$, we obtain
\begin{align*}
\left| (m_{1}+m_{2})\left(\frac{\log 10}{\log \alpha}\right) - n + \left( \frac{\log (d_{1}/9a)}{\log \alpha} \right)  \right| < \frac{56}{10^{m_{1}}\log \alpha}.
\end{align*}
 Thus, applying Lemma \ref{l2}, we obtain:
\begin{align*}
\tau: = \frac{\log 10}{\log \alpha}, \quad \mu (d_1): = \frac{\log (d_{1}/9a)}{\log \alpha}, \quad A := \dfrac{56}{\log\alpha}, \quad B: = 10, \quad \text{and} \quad 1 \leq d_1 \leq 9.
\end{align*}
We choose $M: = 10^{29}$ as the upper bound on $m_{1}+m_{2}$. A quick computation using Mathematica program gives $ q=q_{53} > 6M$, and $\epsilon = 0.0168612 $. Thus, we have that
\begin{align*}
m_{1} \leq \frac{\log (56/\log\alpha)q/\epsilon)}{\log 10} < 34.
\end{align*}
This implies that $m_{1}\leq 34$.
\medskip

\noindent For fixed $0\leq d_{2}< d_{1}\leq 9$ and $1\leq m_{1}\leq 34$, we use (\ref{eq10}) and write
\begin{align*}
\Gamma_2:= m_{2}\log 10-n\log\alpha+\log\left(\frac{d_{1}.10^{m_{1}}-(d_{1}-d_{2})}{9a}\right).
\end{align*}
Now, we rewrite (\ref{eq10}) as 
\begin{align*}
	\left| e^{ \Gamma_2} - 1 \right| < \frac{2}{\alpha^{n}}.
\end{align*}
We notice that $\Gamma_2\neq 0$, since $e^{\Gamma_2}-1=\Lambda_{2}\neq 0$.
For $n > 250$, we have that $$\left| e^{\Gamma_2} - 1 \right| < \displaystyle\frac{1}{2}.$$ Therefore, applying Lemma \ref{l3} gives $$\left|\Gamma_2 \right| < \frac{4}{\alpha^{n}}.$$ 
Replacing $\Gamma_2$ in the above inequality with its value and dividing through by $\log \alpha$, we obtain
\begin{align*}
	\left| m_{2}\left(\frac{\log 10}{\log \alpha}\right) - n + \left( \frac{\log (d_{1}.10^{m_{1}}-(d_{1}-d_{2})/9a)}{\log \alpha} \right)  \right| < \frac{4}{\alpha^{n} \log \alpha}.
\end{align*}
Similarly, by Lemma \ref{l2}, we get the following:
\begin{align*}
	\tau: = \frac{\log 10}{\log \alpha}, \quad \mu (d_1,d_2): = \frac{\log (d_{1}.10^{m_{1}}-(d_{1}-d_{2})/9a)}{\log \alpha}, \quad A := \dfrac{4}{\log\alpha}, \quad B: = \alpha.
\end{align*}

\noindent Since we have that $m_{2}<(m_{1}+m_{2})$, we take $m_{2} < M:= 10^{29}$. Computation using Mathematica program yields $q_{53} > 6M$, and  $ \epsilon = 0.000918645 $. Therefore, we have that
\begin{align*}
n\le \dfrac{\log((4/\log\alpha)q/\epsilon)}{\log\alpha}<200.
\end{align*}
Therefore, $ n\le 200 $, which contradicts the assumption that $ n>250 $. This completes the proof of  Theorem \ref{thm1x}. \qed


\section*{Addresses}

 Department of Mathematics, Makerere University, P.O. Box 7062 Kampala, Uganda


\noindent
Email: \url{mahadi.ddamulira@mak.ac.ug} 

\vspace{0.5cm}

\noindent
$ ^{*} $ Department of Mathematics, Makerere University, P.O. Box 7062 Kampala, Uganda

\noindent 
Email: \url{paul.emong@gmail.com}

\vspace{0.5cm}

\noindent
 Department of Mathematics, Makerere University, P.O. Box 7062 Kampala, Uganda

\noindent
Email:\url{ismail.mirumbe@mak.ac.ug}

\end{document}